\documentclass[1 leqno,11pt,psfig]{amsart}
\usepackage{amssymb, amsmath,amsmath,latexsym,amssymb,amsfonts,amsbsy, amsthm,mathtools,graphicx,CJKutf8,CJKnumb,CJKulem,color}
\usepackage{graphicx,epsfig}
\usepackage{graphicx}
\usepackage{amssymb}
\usepackage{graphicx}
\usepackage{graphicx,epsfig}
\usepackage{amsmath}
\usepackage{mathrsfs}
\usepackage{amssymb}

\usepackage{amssymb,mathrsfs,amsfonts,amsmath,amsbsy,tikz}\usepackage{fontenc}\usepackage{textcomp}

\numberwithin{equation}{section}

\allowdisplaybreaks
\let\a=\alpha

\let\d=\delta

\let\la=\lambda

\let\om=\omega

\def\l{\lambda}
\def\R{\mathbb R}

\def\eqdef{\buildrel\hbox{\footnotesize def}\over =}

\newcommand{\ba}{\begin{array}}
\newcommand{\ea}{\end{array}}

\newcommand{\beq}{\begin{equation}}
\newcommand{\eeq}{\end{equation}}
\newcommand{\ben}{\begin{eqnarray}}
\newcommand{\een}{\end{eqnarray}}
\newcommand{\beno}{\begin{eqnarray*}}
\newcommand{\eeno}{\end{eqnarray*}}

\newtheorem{theorem}{Theorem}[section]
\newtheorem{definition}[theorem]{Definition}
\newtheorem{lemma}[theorem]{Lemma}

\begin{document}
\begin{CJK*}{UTF8}{gkai}
\title[Diffusion and mixing in fluid flow via the resolvent estimate]{Diffusion and mixing in fluid flow via the resolvent estimate}

\author{Dongyi Wei}
\address{School of Math Sciences and BICMR, Peking University, 100871, Beijing, P. R. China}
\email{jnwdyi@163.com}

\date{\today}

\maketitle

\begin{abstract}

In this paper, we first present a Gearhardt-Pr\"uss type theorem with a sharp bound for m-accretive operators. Then we give two applications: (1) give a simple proof of the result proved by Constantin et al. on relaxation enhancement induced by incompressible flows; (2) show that shear flows with a class of Weierstrass functions obey logarithmically fast dissipation time-scales.
\end{abstract}

\section{Introduction}

Let $X$ be a complex Hilbert space. We denote by $\|\cdot\|$ the norm and  by $\langle,\rangle$ the inner product. Let $H$ be a linear operator in $X$ with the domain $D(H)$.  The Hille-Yosida theorem gives a necessary and sufficient condition so that $H$ generates a strongly continuous semigroup $S(t)=e^{tH}$. We say that $S(t)$ satisfies $P(M,\om)$ if
\beno
\|S(t)\|\le Me^{\om t}, \quad t\ge 0.
\eeno

\begin{theorem}[Hille-Yosida theorem]
Let $H$ be a linear operator in $X$ with the domain $D(H)$. Let $\om \in \R, M>0$. Then $H$ generates a strongly continuous semigroup $S(t)=e^{tH}$ satisfying $P(M,\om)$ if and only if
\begin{itemize}

\item[1.] $H$ is closed and $D(H)$ is dense in $X$;

\item[2.]  For all $\lambda>\om$, $\lambda$ belongs to the resolvent set $\rho(H)$ of
$H$, and for all positive integers $n$,
\beno
\big\|(\lambda-H)^{-n}\big\|\le \frac M {(\lambda-\om)^n}.
\eeno
\end{itemize}
\end{theorem}

However, it is not easy to check the second property for all powers of the resolvent.
In the special case of $M=1, \om=0$(i.e., $S(t)$ is a contraction semigroup), it is enough to check that for any $\lambda>0$,
\beno
\big\|(\lambda-H)^{-1}\big\|\le \frac 1 {\lambda}.
\eeno
Gearhart-Pr\"uss theorem gives the semigroup bound via the resolvent estimate.

\begin{theorem}[Gearhart-Pr\"uss theorem]
Let $H$ be a closed operator with a dense domain $D(H)$ generating a strongly continuous semigroup $e^{tH}$. Assume that $\|(z-H)^{-1}\|$
is uniformly bounded for ${\rm{Re}}\ \!z\ge \om$. Then there exists $M>0$
so that $e^{tH}$ satisfies $P(M,\om)$.
\end{theorem}

Let us refer to \cite{EN} for more introductions. \smallskip

Recently, Helffer and Sj\"{o}strand presented a quantitative version of Gearhart-Pr\"uss theorem and gave some interesting applications
to complex Airy operator, complex harmonic oscillator and Fokker-Planck operator \cite{HS}. Motivated by their work,
we first present a Gearhart-Pr\"uss type theorem with a sharp bound for m-accretive operators.
A closed operator $H$ in a Hilbert space $X$ is called m-accretive if the left open half-plane is contained in the resolvent set $\rho(H)$ with
\beno
(H+\lambda)^{-1}\in\mathcal{B}(X),\quad \|(H+\lambda)^{-1}\|\leq ({\rm Re}\lambda)^{-1}\quad \text{for}\ {\rm Re }\lambda>0.
\eeno
Here $\mathcal{B}(X)$ is the set of bounded linear operators on $X$.
An m-accretive operator $H$ is accretive and densely defined (see section V-3.10 in \cite{Kato}), $i.e., {D}(H) $ is dense in $X$ and ${\rm Re} \langle Hf,f\rangle\geq 0$ for $f\in {D}(H),$ and $-H$ is a generator of a semigroup $e^{-tH}$. We denote
\beno
\Psi(H)=\inf\{\|(H-i\lambda)f\|;f\in {D}(H),\ \la\in\R,\  \|f\|=1\}.
\eeno
Let us state the following Gearhart-Pr\"{u}ss type theorem for accretive operators.

\begin{theorem}\label{thm:GP}
Let $H$ be an m-accretive operator in a Hilbert space $X$. Then we have
$\|e^{-tH}\|\leq e^{-t\Psi(H)+{\pi}/{2}}$ for all $t\geq 0$.
\end{theorem}

We will give two applications of Theorem \ref{thm:GP}.\smallskip

The first application of Theorem \ref{thm:GP} is to give a simple proof of the result in \cite{CKRZ} on relaxation enhancement induced by incompressible flows. More precisely, we consider the passive scalar equation
\begin{align}\label{eq5}
\phi^A_t(x,t)+Au\cdot\nabla \phi^A(x,t)-\Delta \phi^A(x,t)=0,\quad \phi(x,0)=\phi_0(x),
\end{align}
in a smooth compact $d$-dimensional Riemannian manifold $M$. Here $\Delta$ is the Laplace-Beltrami operator on $M$, $u$ is a divergence free vector
field. Roughly speaking, a velocity field $u$ is relaxation enhancing if by the diffusive time-scale $O(1)$ arbitrarily much energy is already dissipated for $A$ large enough. The main result of \cite{CKRZ} characterizes relaxation enhancing
flows in terms of the spectral properties of the operator $ u\cdot\nabla.$ Precisely, $u$ is relaxation enhancing if and
only if the operator $ u\cdot\nabla$ has no nontrivial eigenfunctions in $\dot{H}^1(M).$ The proof of this result is based on the so-called
RAGE theorem. Thanks to Theorem \ref{thm:GP}, the proof of relaxation enhancing can be reduced to the resolvent estimate of the operator $Au\cdot\nabla-\Delta, $ avoiding the use of the RAGE theorem. See section 3 for more details.\smallskip

For $A>0,$ by time rescaling $ \tau=At$, \eqref{eq5} becomes
\begin{align}\label{eq7}
\phi_{\tau}+u\cdot\nabla \phi-\nu\Delta \phi=0,\ \ \phi(x,0)=\phi_0(x),
\end{align}
with $\nu=1/A.$ Our second application focus on the case when $M=\mathbb{T}^2$ and $u$ is a shear flow. More precisely, we study the decay estimates in time of the linear evolution semigroup $S_{\nu}(t):L^2(\mathbb{T}^2)\to L^2(\mathbb{T}^2),$ with $\nu>0$ a positive
parameter, generated by the drift-diffusion scalar equation \begin{align}\label{eq1}
\partial_t f-\nu\Delta f+u(y)\partial_xf=0,\ \ (x,y)\in\mathbb{T}^2,\ t>0,
\end{align} and of its hypoelliptic counterpart $R_{\nu}(t):L^2(\mathbb{T}^2)\to L^2(\mathbb{T}^2),$ generated by \begin{align}\label{eq2}
\partial_t f-\nu\partial_y^2 f+u(y)\partial_xf=0,\ \ (x,y)\in\mathbb{T}^2,\ t>0.
\end{align}
The case of general shear flows $(u(y),0)$ with a finite number of critical points was treated in \cite{BZ}, where the enhanced dissipation time-scale was proved to be $O(\nu^{-p}),$ for $p=p(n_0)=\frac{n_0+1}{n_0+3}\geq \frac{1}{3},$ where $n_0$ denotes the the maximal order of vanishing of $u'$ at the critical points.
The proof used the hypocoercivity method in \cite{Vi}. See also \cite{BH} for an interesting application of the enhanced dissipation and related work \cite{KZ}.

Our proof is based on Theorem \ref{thm:GP}. To this end, we study the resolvent estimate of $H=-\partial_y^2+iu(y)$ with $X=L^2(\mathbb{T})$ and ${D}(H)=H^2(\mathbb{T})$ in section 4. Set $\Psi_1(u)=\Psi(H)$ and
$$\Psi_0(u)=\inf\big\{\|Hf\|_{L^2};f\in {D}(H),\  \|f\|_{L^2}=1\big\}.$$
Then $\Psi_1(u)=\inf_{\l\in\R}\Psi_0(u-\l).$
We will give a lower bound of $\Psi_0(u)$ and $\Psi_1(u)$ in terms of the following quantities:
\beno
&&\omega_0(\delta,u)=\inf_{x,c\in\R}\int_{x-\d}^{x+\d}|\psi(y)-c|^2dy,\\
&&\omega_1(\delta,u)=\inf_{c\in\R}\omega_0(\delta,u-c)=\inf_{x,c_1,c_2\in\R}\int_{x-\d}^{x+\d}|\psi(y)-c_1-c_2y|^2dy.
\eeno
Here $u(y)=\psi'(y)$ and we identify a function on $\mathbb{T}$ with a $2\pi$-periodic function on $\R$. Notice that $\psi(y)$ may not be a periodic function. The quantities $\omega_0,\omega_1 $ are well defined since their values do not change if $\psi$ is added by a constant. We will give the dependence of $p$ on $u$ via the quantity $\omega_1(\delta,u) $.
More precisely, if $\omega_1(\delta,u)\geq C_1\d^{2\a+3} $ for $\d\in(0,1)$ and some constants $\a>0,\ C_1>0,$ then the enhanced dissipation time-scale is $O(\nu^{-\frac{\a}{\a+2}}).$

To show the effectiveness of this criterion, we will discuss the case when $u(y)=\\ \sum\limits_{n=1}^{\infty}a_n\sin(3^ny)$ is a Weierstrass function. Our result is
stated as follows.

\begin{itemize}
\item If $a_n\in\R,\ 3^{-n\a}\leq |a_n|\leq C3^{-n\a}$ for some $\a\in(0,1),$ and $1\leq |a_n|/|a_{n+1}|\leq 3$, then the enhanced dissipation time-scale is $O(\nu^{-\frac{\a}{\a+2}})$;
\item If $a_n\in\R,\ n^{-\a}\leq |a_n|\leq Cn^{-\a}$ for some $\a\in(1,2),$ and $1\leq |a_n|/|a_{n+1}|\leq 3$, then the enhanced dissipation time-scale is $O(|\ln \nu|^{\a}).$
\end{itemize}

Recently, Coti Zelati, Delgadino and Elgindi proved the enhanced dissipation time-scale $O(|\ln \nu|^{2})$ in \eqref{eq7} for all contact Anosov flows on a smooth $2d+1$ dimensional connected compact Riemannian manifold \cite{CDM}. Their proof is based on the knowledge of mixing decay rates.
Our result shows that similar phenomenon also happens for a class of shear flows.

Let us also mention some important progress on the enhanced dissipation  of the linearized Navier-Stokes equations around shear flows such as Couette flow and Kolmogorov flow \cite{BW,BGM,GNRS,IMM,LWZ,WZ,WZZ}.
\medskip

Throughout this paper, we denote by $C$ a constant independent
of $A,\nu,t$, which may be different from line to line.

\section{Proof of Gearhart-Pr\"{u}ss type theorem}

In this section, we prove Theorem \ref{thm:GP}. The proof is partially motivated by \cite{HS}.

\begin{proof}
Let $\Psi=\Psi(H)$. Since ${D}(H) $ is dense in $X$, we only need to prove that
\begin{align}\label{eqA}\|e^{-tH}f\|\leq e^{-t\Psi+{\pi}/{2}}\|f\|,\ \ \forall\ f\in {D}(H),\ t\geq 0.
\end{align}
For $f\in {D}(H),\ t\geq 0,$ let $g(t)=\|e^{-tH}f\|^2$. Since $H$ is accretive, $g(t)$ is decreasing for $t\geq 0,$ and we only need to prove \eqref{eqA} for $t\Psi>{\pi}/{2}. $ In this case, $\Psi>0$. We denote
$$t_1=\frac{\pi}{4\Psi},\ t_2=t-\frac{\pi}{4\Psi},\ t_3=t+\frac{\pi}{4\Psi},\ l=t+\frac{\pi}{2\Psi}.$$
For $\chi\in C^1[0,l],\ \chi(0)=\chi(l)=0,$ set $f_1(s)=\chi(s)e^{-sH}f,\ f_2(s)=\chi'(s)e^{-sH}f$. Then $\partial_tf_1+Hf_1=f_2$ in $[0,l]$.
Take Fourier transform in $t$: $ \widehat{f_j}(\la)=\int_0^lf_j(s)e^{-i\la s}ds$, for $j=1,2,\ \la\in\R.$ Then $\widehat{f_2}(\lambda)=(i\la+H)\widehat{f_1}(\lambda)$.
By the definition of $\Psi$, we have $\|\widehat{f_2}(\lambda)\|\geq \Psi\|\widehat{f_1}(\lambda)\|.$
We use Plancherel's Theorem to conclude
$$\|f_2\|_{L^2([0,l],X)}=(2\pi)^{-\frac{1}{2}}\|\widehat{f_2}\|_{L^2(\R,X)}
\geq(2\pi)^{-\frac{1}{2}}\Psi\|\widehat{f_1}\|_{L^2(\R,X)}=\Psi\|f_1\|_{L^2([0,l],X)}.$$

By the definitions of $f_1,\ f_2,\ g,$ the above inequality becomes
$$
\int_0^l\chi'(s)^2g(s)ds\geq \Psi^2\int_0^l\chi(s)^2g(s)ds.
$$
Now we choose $\chi$ as follows
$$\chi(s)=\left\{\ba{ll}\sin\Psi s,&0\leq s\leq t_1,\\ e^{\Psi s-\pi/4}/\sqrt{2},&t_1\leq s\leq t_2,\\ e^{\Psi l-\pi}\sin(\Psi(l-s) ),&t_2\leq s\leq l.\ea\right.$$
Set $h(s)=\chi'(s)^2-\Psi^2\chi(s)^2$. Then $\int_0^lh(s)g(s)ds\geq 0,$ and $$h(s)=\left\{\ba{ll}\Psi^2\cos(2\Psi s),&0\leq s\leq t_1,\\ 0,&t_1\leq s\leq t_2,\\ \Psi^2e^{2\Psi l-2\pi}\cos(2\Psi(l-s) ),&t_2\leq s\leq l.\ea\right.$$
Therefore, $h(s)\geq 0$ for $0\leq s\leq t_1 $ or $t_3\leq s\leq l, $ $h(s)\leq 0$  for $t_2\leq s\leq t_3. $ Since $g$ is decreasing, we have $h(s)g(s)\leq h(s)g(0)$ for $0\leq s\leq t_1 $, $h(s)g(s)\leq h(s)g(t)$ for $t_2\leq s\leq t,$ $h(s)g(s)\leq h(s)g(t_3)$ for $t\leq s\leq l, $ and
\begin{align*}
0&\leq\int_0^lh(s)g(s)ds=\int_0^{t_1}h(s)g(s)ds+\int_{t_2}^{l}h(s)g(s)ds\\ &\leq\int_0^{t_1}h(s)g(0)ds+\int_{t_2}^{t}h(s)g(t)ds+\int_{t}^{l}h(s)g(t_3)ds\\
&=\frac{\Psi}{2}g(0)-\frac{\Psi}{2}e^{2\Psi l-2\pi}g(t)+0.
\end{align*}
Therefore, $g(t)\leq e^{-2\Psi l+2\pi}g(0)$, which implies that
$$\|e^{-tH}f\|\leq e^{-\Psi l+\pi}\|f\|=e^{-\Psi t+\pi/2}\|f\|.$$

This completes the proof.\end{proof}

\section{Diffusion and mixing in fluid flow}

Let us recall the following definition from \cite{CKRZ}.

\begin{definition}Let $M$ be a smooth compact Riemannian manifold. The
incompressible flow $u$ on $M$ is called relaxation enhancing if for every $ \tau> 0$ and
$\d> 0$, there exist $A(\tau, \d)$ such that for any $ A> A(\tau, \d)$ and any $ \phi_0\in L^2(M),\ \|\phi_0\|_{L^2(M)}=1,$ $$\|\phi^A(\cdot,\tau)-\overline{\phi}\|_{L^2(M)}<\d,$$ where $\phi^A(x,t)$ is the solution of \eqref{eq5} and $\overline{\phi} $ the average of $\phi_0 $.
\end{definition}

 We take $X=\big\{f\in L^2(M)|\int_M f=0\big\}$ the subspace of
mean zero functions, $H=H_A=-\Delta+Au\cdot\nabla$ with ${D}(H)=H^2(M)\cap X .$ Set $\Psi_2(A)=\Psi(H_A).$ Our result is as follows.

\begin{theorem}
Let $M$ be a smooth compact Riemannian manifold. A
continuous incompressible
flow $u$ is relaxation-enhancing
if and only if the operator $u\cdot\nabla $ has no eigenfunctions in $H^1(M),$ other than
the constant function.
\end{theorem}

\begin{proof}The proof of the first part is the same as in \cite{CKRZ}. First of all, we have $$\partial_t\|\phi^A(t)\|_{L^2}^2=\langle\phi^A_t,\phi^A\rangle+\langle\phi^A,\phi^A_t\rangle=-2\|\nabla\phi^A(t)\|_{L^2}^2. $$
If $u\cdot\nabla $ has non constant eigenfunctions in $H^1(M),$ then $u\cdot\nabla $ has non constant eigenfunctions in $H^1(M)\cap X.$
Assume that the initial datum $\phi_0\in H^1(M)\cap X$ for \eqref{eq5} is an eigenvector of $u\cdot\nabla $ corresponding to an eigenvalue $i\la$,
normalized so that $\|\phi_0\|_{L^2}=1,$ then $\overline{\phi}=0,\ \la\in\R.$
Take the inner product of \eqref{eq5} with $\phi_0$, we
arrive at
$$\partial_t\langle\phi^A(t),\phi_0\rangle=-iA\l\langle\phi^A(t),\phi_0\rangle-\langle\Delta\phi^A(t),\phi_0\rangle.$$
This along with the assumption $\phi_0\in H^1(M)$ leads to
$$|\partial_t(e^{iA\l t}\langle\phi^A(t),\phi_0\rangle)|=|\langle\nabla\phi^A(t),\nabla\phi_0\rangle|\leq \frac{1}{2}(\|\nabla\phi^A(t)\|_{L^2}^2+\|\nabla\phi_0\|_{L^2}^2).$$
Note that $\int_0^{\tau}\|\nabla\phi^A(t)\|_{L^2}^2dt=(\|\phi_0\|_{L^2}^2-\|\phi^A(t)\|_{L^2}^2)/2\leq1/2$. Then for $0<t\leq\tau=(2\|\nabla\phi_0\|_{L^2}^2)^{-1}$,
we have $|\langle\phi^A(t),\phi_0\rangle|\geq 1/2. $ Thus, $\|\phi^A(\tau)\|_{L^2}\geq 1/2$ uniformly in $A,$ and $u$ is not relaxation-enhancing.

Now we prove the converse, we first claim that $\lim\limits_{A\to+\infty}\Psi_2(A)=+\infty $ implies relaxation-enhancing. In fact, since $\phi^A(\cdot,\tau)-\overline{\phi}\in X $ and $\phi^A(\cdot,\tau)-\overline{\phi}=e^{-\tau H_A}(\phi_0-\overline{\phi}),$ by Theorem \ref{thm:GP}, we have $$\|\phi^A(\cdot,\tau)-\overline{\phi}\|_{L^2(M)}\leq e^{-\tau\Psi_2(A)+{\pi}/{2}}\|\phi_0-\overline{\phi}\|_{L^2(M)}\leq e^{-\tau\Psi_2(A)+{\pi}/{2}}.$$
If $\lim\limits_{A\to+\infty}\Psi_2(A)=+\infty $, then we can find $A(\tau, \d)$ such that for any $ A> A(\tau, \d),$ we have $\Psi_2(A)>(\pi/2-\ln\d)/\tau,$ and thus $\|\phi^A(\cdot,\tau)-\overline{\phi}\|_{L^2(M)}<\d.$

Next we claim that $\liminf\limits_{A\to+\infty}\Psi_2(A)<+\infty$ implies that $u\cdot\nabla $ has a nonzero eigenfunction in $H^1(M)\cap X.$
In fact, in this case, there exists $A_n\to+\infty,\ C_0\in\R$ such that $\Psi_2(A_n)<C_0 $ and there exists $\la_n\in\R,\ f_n\in X$ such that $\|f_n\|_{L^2(M)}=1,\ \|(H_{A_n}-i\la_n)f_n\|_{L^2(M)}<C_0. $
Then $$\|\nabla f_n\|_{L^2(M)}^2={\rm Re}\langle f_n,g_n\rangle\leq \|f_n\|_{L^2(M)}\|g_n\|_{L^2(M)}<C_0,$$
here $g_n=(H_{A_n}-i\la_n)f_n$. Thus, the sequence $\{f_n\}$ is bounded in $H^1(M)$ and there exists a subsequence of $\{f_n\}$ (still denoted by $\{f_n\}$) and $f_0\in H^1(M),$ such that $f_n\to f_0 $ strongly in $L^2(M).$
Therefore, $\|f_0\|_{L^2(M)}=1,\ f_0\in X.$ For $f\in H^1(M),$ we have
$$\langle g_n,f\rangle=\langle \nabla f_n,\nabla f\rangle+A_n\langle u\cdot\nabla f_n,f\rangle-i\la_n\langle f_n,f\rangle,$$
and
$$\langle u\cdot\nabla f_n,f\rangle-i\frac{\la_n}{A_n}\langle f_n,f\rangle=\frac{\langle g_n,f\rangle-\langle \nabla f_n,\nabla f\rangle}{A_n}\to 0,$$
as $n\to +\infty,$ here we used $A_n\to +\infty$ and
\begin{align*}|\langle g_n,f\rangle-\langle \nabla f_n,\nabla f\rangle|&\leq \|g_n\|_{L^2(M)}\|f\|_{L^2(M)}+\|\nabla f_n\|_{L^2(M)}\|\nabla f\|_{L^2(M)}\\ &\leq C_0\|f\|_{L^2(M)}+C_0^{\frac{1}{2}}\|\nabla f\|_{L^2(M)}.
\end{align*}
Moreover,
 $$\langle u\cdot\nabla f_n,f\rangle=-\langle  f_n,u\cdot\nabla f\rangle\to -\langle  f_0,u\cdot\nabla f\rangle=\langle u\cdot\nabla f_0,f\rangle,$$
 and $\langle f_n,f\rangle\to \langle f_0,f\rangle$ as $n\to +\infty$. Therefore,
 $$\lim_{n\to +\infty}i\frac{\la_n}{A_n}\langle f_n,f\rangle=\langle u\cdot\nabla f_0,f\rangle.$$
 If we take $f=f_0$ then we have $\langle f_n,f\rangle\to \langle f_0,f\rangle=\langle f_0,f_0\rangle=1\neq 0 $ and $i\dfrac{\la_n}{A_n}\to \langle u\cdot\nabla f_0,f_0\rangle\eqdef i\l, $ as $n\to +\infty$. Therefore, for every $f\in H^1(M),$ we have
 $$i\la\langle f_0,f\rangle=\lim_{n\to +\infty}i\frac{\la_n}{A_n}\langle f_n,f\rangle=\langle u\cdot\nabla f_0,f\rangle.$$
 Since $H^1(M)$ is dense in $L^2(M)$, we have $i\la f_0=u\cdot\nabla f_0$, $f_0\neq 0.$ This proves our claim.

With the above two claims, we prove the converse part.
\end{proof}

Compared with \cite{CKRZ}, we don't need to assume $u$ to be Lipschitz continuous and our proof is more easier. As in \cite{CKRZ}, with a slight modification, we can prove a more general result. Let $ \Gamma$ be
a self-adjoint, positive, unbounded operator with a discrete spectrum on a
separable Hilbert space $H$: Let $0 < \la_1\leq\la_2 \leq \cdots$ be the eigenvalues of $ \Gamma$,
and $e_j$ the corresponding orthonormal eigenvectors forming a basis in $H$. The
(homogeneous) Sobolev space $H^m(\Gamma)$ associated with $ \Gamma$ is formed by all vectors $\psi=\sum_jc_je_j$ such that $\|\psi\|_m^2=\sum_j\la_j^m|c_j|^2<+\infty.$ Note that $H^2(\Gamma)={D}(\Gamma).$ Let $L$ be a symmetric operator
such that $H^1(\Gamma)\subseteq{D}(L) $ and $\|L\psi\|_0\leq C\|\psi\|_1 $ for $\psi\in H^1(\Gamma). $ Consider a solution $\phi^A(t)$ of the Bochner differential equation
\begin{align}\label{eqL}
\partial_t\phi^A(t)=iAL\phi^A(t)-\Gamma\phi^A(t),\quad \phi^A(0)=\phi_0.
\end{align}

\begin{theorem}
For $\Gamma,\ L$ satisfying the above conditions, the following two statements are equivalent:
\begin{itemize}
\item For every $\tau> 0$ and $\d> 0$,
there exist $A(\tau, \d)$ such that for any $ A> A(\tau, \d)$ and any $ \phi_0\in H,\ \|\phi_0\|_{0}=1,$ the solution $ \phi^A(t) $ of the equation \eqref{eqL} satisfies $\|\phi^A(\tau)\|_{0}<\d.$ \item The operator $L$ has no eigenvectors lying in $H^1(\Gamma)$. \end{itemize}
\end{theorem}

Here we don't assume that $\| e^{iLt}\psi\|_1\leq B(t)\|\psi\|_1$. Therefore, our result is applicable to the example given in \cite{CKRZ}: $H=L^2(0,1),\ \Gamma f(x)=\sum_n e^{n^2}\widehat{f}(n)e^{2\pi i nx},\ Lf(x)=xf(x).$

\section{Resolvent estimate for the shear flows}

In this section, we give the resolvent estimate of the operator $ H=H_{(u)}=-\partial_y^2+iu(y).$
We start with a few basic observations concerning the operator $H_{(u)}.$ As is well
known, the operator $H_{(0)}=-\partial_y^2$ is self-adjoint in $L^2(\mathbb{T})$ with a compact resolvent,
and its spectrum is a sequence of eigenvalues $\{\la_n^0\}_{n\in \mathbb{N}}$, where $\la_{0}^0=0,\ \la_{2n}^0=\la_{2n-1}^0=n^2.$ By the classical perturbation theory \cite{Kato}, it follows that $H_{(u)}$ has a compact resolvent for
any $u\in C(\mathbb{T},\R),$ and that its spectrum is again a sequence of eigenvalues $\{\la_n^{(u)}\}_{n\in \mathbb{N}}$, with ${\rm Re}(\la_n^{(u)})\to+\infty$ as $n\to+\infty.$

Since ${\rm Re}\langle Hf,f\rangle_{L^2}=\|\partial_y f\|_{L^2}^2\geq 0$ for $f\in {D}(H)=H^2(\mathbb{T})$, $H$ is accretive, ${\rm Re}(\la_n^{(u)})\geq 0,$ and $\|(H+\la)u\|_{L^2}\|u\|_{L^2}\geq {\rm Re}\langle (H+\la)u,u\rangle_{L^2}\geq {\rm Re}\langle \la u,u\rangle_{L^2}=({\rm Re} \la)\|u\|_{L^2}^2,\ \|(H+\la)u\|_{L^2}\geq ({\rm Re} \la)\|u\|_{L^2}$ for ${\rm Re}\la>0, $ which implies that $H$ is m-accretive.

For $\la\in\R$, $H-i\la$ is invertible if and only if
$$
\inf\{\|(H-i\la)f\|_{L^2};f\in {D}(H),\quad \|f\|_{L^2}=1\}>0.
$$
If $\Psi(H)>0 $ then $H-i\la$ is invertible for all $\la\in\R$, and
$$\Psi(H)=\left(\sup_{\la\in\R}\|(H-i\la)f\|\right)^{-1}=\inf\big\{\|(H-i\la)f\|_{L^2};f\in {D}(H),\ \l\in\R,\  \|f\|_{L^2}=1\big\}.
$$
Thus, our definition of $\Psi(H) $ is the same as in \cite{GGN}. We first give a lower bound of $\Psi_0(u)$ in terms of $ \omega_0(\delta,u)$. Then lower bound of $\Psi_1(u)$ follows by minimizing $\la.$ Recall that $$\Psi_0(u)=\inf\{\|Hf\|_{L^2};f\in {D}(H),\  \|f\|_{L^2}=1\}.$$

The following lemma shows the existence of the minimizer.

\begin{lemma}\label{Lem: 2}
If $ \mu=\Psi_0(u),$ then there exists $0\neq f\in {D}(H)$ so that
 $Hf=\mu\overline{f}$.
\end{lemma}

\begin{proof}By the definition of $ \mu=\Psi_0(u)$, we have $\|Hg\|_{L^2}\geq \mu\|g\|_{L^2}$ for all $g\in {D}(H),$ and we can take $f_n\in {D}(H),\ \|f_n\|_{L^2}=1$ such that $\|Hf_n\|_{L^2}\to\mu $ as $n\to\infty$. Then the sequence $\{f_n\}$ is bounded in $H^2(\mathbb{T})$, and there exists a subsequence of $\{f_n\}$ (still denoted by $\{f_n\}$) and $f_0\in {D}(H),$ such that $f_n\to f_0 $ strongly in $L^2(\mathbb{T})$ and $Hf_n\rightharpoonup Hf_0$ weakly  in $L^2(\mathbb{T})$ as $n\to\infty$.
Therefore, $\|f_0\|_{L^2}=1,\ \|Hf_0\|_{L^2}\leq\mu.$ If $\mu=0$, then we can take $ f=f_0.$ If $\mu>0,$  we have for all $g\in {D}(H),\ t\in\R,\ \|Hf_0+tHg\|_{L^2}\geq \mu\|f_0+tg\|_{L^2},$ and the equality holds at $t=0,$ therefore,
$$
0=\left.\frac{d}{dt}\right|_{t=0}(\|Hf_0+tHg\|_{L^2}^2-\mu^2\|f_0+tg\|_{L^2}^2)=2{\rm Re}\langle Hf_0,Hg\rangle-2\mu^2{\rm Re}\langle f_0,g\rangle,
$$
and we also have $0=2{\rm Re}\langle Hf_0,iHg\rangle-2\mu^2{\rm Re}\langle f_0,ig\rangle$. Thus, $2\langle Hf_0,Hg\rangle=2\mu^2\langle f_0,g\rangle. $ Set $Hf_0=\mu g_0 $, then $\langle g_0,Hg\rangle=\mu\langle f_0,g\rangle $ for all $g\in {D}(H)$. This implies that $g_0\in {D}(H^*), $ and $H^*g_0=\mu f_0$. Here $H^*=-\partial_y^2-iu(y)$ and ${D}(H^*)=H^2(\mathbb{T}).$ Therefore, $\overline{g_0}\in {D}(H) $ and  $H\overline{g_0}=\mu \overline{f_0}$. Since $f_0+\overline{g_0}\neq 0$ or $f_0-\overline{g_0}\neq 0$, we can take $f=f_0+\overline{g_0}$ or $i(f_0-\overline{g_0})$.
\end{proof}

Now we need to study the equation $Hf=\mu\overline{f}$. Set $u(y)=\psi'(y),\ \psi(y)\in\R$ for $y\in \R.$ Now we can define $\omega_0(\delta,u),\ \omega_1(\delta,u)$ as in section 1. Recall that
$$\omega_0(\delta,u)=\inf_{x,c\in\R}\int_{x-\d}^{x+\d}|\psi(y)-c|^2dy.$$

 \begin{lemma}\label{Lem: 3}
 If $0\neq f\in {D}(H),\ Hf=\mu\overline{f},\ \mu\geq 0,\ \delta>0,$ then $\sqrt{\mu}\geq\dfrac{\pi}{2\d}$ or $36\sqrt{\mu}\tan(\sqrt{\mu}\d)\geq \omega_0(\delta,u).$
 \end{lemma}

\begin{proof}If $\mu=0$, then $0={\rm Re}\langle Hf, f\rangle_{L^2}=\|\partial_y f\|_{L^2}^2,$ $f$ is constant, $u\equiv0, $ $\psi$ is a constant, $\omega_0(\delta,u)=0,$ and the result is true. Now we assume $\mu>0.$

In this case, we have $f\in C^2(\mathbb{T})$. We can normalize $\|f\|_{L^{\infty}}=1,$ and assume $|f(y_0)|=1$ for some $y_0\in\mathbb{R}.$
Set $a=\sup\{y|f(y)=0,y<y_0\},\ b=\inf\{y|f(y)=0,y>y_0\}$. Then $-\infty\leq a<y_0<b\leq +\infty $ and $f\neq 0$ in $(a,b)$, $f(a)=0$ if $a>-\infty,$ $f(b)=0$ if $b<+\infty.$ Now we can find $g\in C^2(a,b)$ such that $ f=e^g$ in $(a,b)$, then $Hf=(-g''-g'^2+iu)f$. Set $g=\rho+i\theta,\ \rho,\theta\in\R$, then $\overline{f}=e^{-2i\theta}f$ and the equation $Hf=\mu\overline{f} $ in $(a,b)$ can be written as $-g''-g'^2+iu=\mu e^{-2i\theta}$ or $$-\rho''-\rho'^2+\theta'^2=\mu \cos{2\theta},\ \ \ -\theta''-2\rho'\theta'+u=-\mu \sin{2\theta}.$$ As $\|f\|_{L^{\infty}}=1,\ |f(y_0)|=1,$ we have $\rho\leq 0$ in $(a,b)$, $\rho(y_0)=0,$ and $\rho'(y_0)=0. $

We first give the lower bound of $y_0-a$ and $b-y_0.$ By the standard theory of
ODE, if $a>-\infty$, then $\lim\limits_{y\to a_+}\rho(y)=-\infty$, while if $b<+\infty$, then $\lim\limits_{y\to b_-}\rho(y)=-\infty$.

Since $\rho''+\rho'^2+\mu=\theta'^2+2\mu(\sin{\theta})^2\geq 0, $ set $\rho_1=\arctan\dfrac{\rho'}{\sqrt{\mu}}$, then $\rho_1(y_0)=0,\ \dfrac{\rho_1'}{\sqrt{\mu}}+1=\dfrac{\rho''}{\rho'^2+\mu}+1\geq 0,$ and $\rho_1(z)\geq\rho_1(y)+(y-z)\sqrt{\mu} $ for $a<y<z<b.$

For $y\in (a,b)$ if $b<y+(\rho_1(y)+\pi/2)/\sqrt{\mu},$ then $$\inf\limits_{(y,b)}\rho_1\geq\rho_1(y)+(y-b)\sqrt{\mu}>-\dfrac{\pi}{2},\ \inf\limits_{(y,b)}\rho'>-\infty,\ \inf\limits_{(y,b)}\rho>-\infty,$$ a contradiction, therefore $b\geq y+(\rho_1(y)+\pi/2)/\sqrt{\mu}$. Similarly, $a\leq y+(\rho_1(y)-\pi/2)/\sqrt{\mu}.$ In particular, $a\leq y_0-\dfrac{\pi}{2\sqrt{\mu}}<y_0+\dfrac{\pi}{2\sqrt{\mu}}\leq b.$

Now we estimate $|\rho'(y)|.$ For $y\in (a,b)$, we have $y_1=y+\rho_1(y)/\sqrt{\mu}\in (a,b),$ \begin{align*}\rho(y_1)-\rho(y)=\int_{y}^{y_1}\rho'(z)dz=\int_{y}^{y_1}\sqrt{\mu}\tan\rho_1(z)dz\\
\geq\int_{y}^{y_1}\sqrt{\mu}\tan(\rho_1(y)+(y-z)\sqrt{\mu})dz{=\ln\frac{1}{\cos\rho_1(y)}}.\end{align*}
Since $\rho(y_1)\leq 0$, we have $e^{\rho(y)}\leq\cos\rho_1(y)$. On the other hand, if $|y-y_0|<\dfrac{\pi}{2\sqrt{\mu}}$, then \begin{align*}\rho(y_0)-\rho(y)=\int_{y}^{y_0}\rho'(z)dz=\int_{y}^{y_0}\sqrt{\mu}\tan\rho_1(z)dz\\
\leq\int_{y}^{y_0}\sqrt{\mu}\tan(\rho_1(y_0)+(y_0-z)\sqrt{\mu})dz=\ln\frac{1}{\cos((y-y_0)\sqrt{\mu})}.\end{align*}
Here we used $\rho_1(y_0)=\rho(y_0)=0$. Therefore, $\cos((y-y_0)\sqrt{\mu})\leq e^{\rho(y)}\leq\cos\rho_1(y). $ Since $(y-y_0)\sqrt{\mu},\ \rho_1(y)\in (-\pi/2,\pi/2) $ , we have $|y-y_0|\sqrt{\mu}\geq|\rho_1(y)| $ and $|\rho'(y)|=\sqrt{\mu}\tan|\rho_1(y)|\leq \sqrt{\mu}\tan(|y-y_0|\sqrt{\mu}).$ Now if $\sqrt{\mu}<\dfrac{\pi}{2\d}$, then $\d<\dfrac{\pi}{2\sqrt{\mu}}$ and
\begin{align*}&\int_{y_0-\d}^{y_0+\d}|\theta'(z)|^2dz\leq\int_{y_0-\d}^{y_0+\d}(\rho''+\rho'^2+\mu)dz\\
\leq&\rho'|_{y_0-\d}^{y_0+\d}+\int_{y_0-\d}^{y_0+\d}({\mu}\tan^2(|z-y_0|\sqrt{\mu})+\mu)dz\\ \leq&2\sqrt{\mu}\tan(\d\sqrt{\mu})+2\sqrt{\mu}\tan(\d\sqrt{\mu})=4\sqrt{\mu}\tan(\d\sqrt{\mu}).\end{align*}
Here we used $ |\rho'(y)|\leq \sqrt{\mu}\tan(|y-y_0|\sqrt{\mu})=\sqrt{\mu}\tan(\d\sqrt{\mu})$ for $y=y_0\pm\d$.

Now we estimate $\omega_0(\delta,u).$ Since $u(y)=\psi'(y)$, we have $-\theta''-2\rho'\theta'+\psi'=-\mu \sin{2\theta},$ and $ \psi(y)-\theta'(y)-c=\int_{y_0}^y(2\rho'\theta'-\mu \sin{2\theta})dz$ for $c=\psi(y_0)-\theta'(y_0).$ If $y_0<y<y_0+\d$, then
\begin{align*}&|\psi(y)-\theta'(y)-c|\leq\int_{y_0}^y\frac{2\rho'^2+\theta'^2+2\mu}{\sqrt{2}}dz
\leq\int_{y_0}^y\frac{2\rho'^2+(\rho''+\rho'^2+\mu)+2\mu}{\sqrt{2}}dz\\
=&\frac{\rho'|_{y_0}^{y}}{\sqrt{2}}+\frac{3}{\sqrt{2}}\int_{y_0}^y(\rho'^2+\mu)dz\leq \frac{\rho'({y})}{\sqrt{2}} +\frac{3}{\sqrt{2}}\int_{y_0}^y({\mu}\tan^2(|z-y_0|\sqrt{\mu})+\mu)dz\\ \leq&\sqrt{\mu/2}\tan((y-y_0)\sqrt{\mu})+3\sqrt{\mu/2}\tan((y-y_0)\sqrt{\mu})=2\sqrt{2\mu}\tan((y-y_0)\sqrt{\mu}).\end{align*}
Similarly, if $y_0-\d<y<y_0$, then $|\psi(y)-\theta'(y)-c|\leq 2\sqrt{2\mu}\tan(|y-y_0|\sqrt{\mu}).$ Therefore,
 \begin{align*}&\omega_0(\delta,u)\leq \int_{y_0-\d}^{y_0+\d}|\psi(y)-c|^2dy\leq3\int_{y_0-\d}^{y_0+\d}|\theta'(z)|^2dz+\frac{3}{2}\int_{y_0-\d}^{y_0+\d}
|\psi(y)-\theta'(y)-c|^2dy\\
\leq&3\cdot4\sqrt{\mu}\tan(\d\sqrt{\mu})+\frac{3}{2}\int_{y_0-\d}^{y_0+\d}8\mu\tan^2(|y-y_0|\sqrt{\mu})dy\\ =&12\sqrt{\mu}\tan(\d\sqrt{\mu})+24(\sqrt{\mu}\tan(\d\sqrt{\mu})-\mu\d)\leq36\sqrt{\mu}\tan(\d\sqrt{\mu}).\end{align*}

This completes the proof.\end{proof}

Set $\varphi:[0,\pi/2)\to[0,+\infty),\ \varphi(x)=36x\tan x$. Then $\varphi $ is a one to one increasing function and we denote $\varphi^{-1}:[0,+\infty)\to[0,\pi/2)$ to be the inverse function.

\begin{lemma}\label{Lem: 4}
For  $\d>0$,  we have
\beno
\Psi_0(u)\geq(\varphi^{-1}(\delta\omega_0(\delta,u))/{\d})^2,\quad \Psi_1(u)\geq(\varphi^{-1}(\delta\omega_1(\delta,u))/{\d})^2.
\eeno
\end{lemma}

\begin{proof} Let $\mu=\Psi_0(u),$ by Lemma \ref{Lem: 2}, there exists $0\neq f\in {D}(H)$ such that
 $Hf=\mu\overline{f}$. By Lemma \ref{Lem: 3}, we have $\sqrt{\mu}\geq\dfrac{\pi}{2\d}$ or $36\sqrt{\mu}\tan(\sqrt{\mu}\d)\geq \omega_0(\delta,u).$ Therefore, $\sqrt{\mu}\d\geq\dfrac{\pi}{2}$ or $\varphi(\sqrt{\mu}\d)=36\sqrt{\mu}\d\tan(\sqrt{\mu}\d)\geq \d\omega_0(\delta,u).$

 Since $\varphi(x)=36x\tan x$ is a one to one increasing function, we have $\sqrt{\mu}\d\geq\dfrac{\pi}{2}$ or $\sqrt{\mu}\d\geq \varphi^{-1}( \d\omega_0(\delta,u)).$ As $ \varphi^{-1}( \d\omega_0(\delta,u))<\pi/2, $ $\sqrt{\mu}\d\geq \varphi^{-1}( \d\omega_0(\delta,u))$ is always true, and $\sqrt{\mu}\geq \varphi^{-1}( \d\omega_0(\delta,u))/\d,\ \Psi_0(u)= (\sqrt{\mu})^2\geq (\varphi^{-1}( \d\omega_0(\delta,u))/\d)^2.$ Now we have\begin{align*}\Psi_1(u)=\inf_{\l\in\R}\Psi_0(u-\l)\geq \inf_{\l\in\R}(\varphi^{-1}( \d\omega_0(\delta,u-\l))/\d)^2\\=(\varphi^{-1}( \d\inf_{\l\in\R}\omega_0(\delta,u-\l))/\d)^2=(\varphi^{-1}( \d\omega_1(\delta,u))/\d)^2.\end{align*}This completes the proof.
 \end{proof}

 \section{Enhanced dissipation for shear flows}

As in \cite{BZ}, let $L_{k,\nu}=iku-\nu(\partial_y^2-|k|^2),\ R_{k,\nu}=iku-\nu\partial_y^2,$ be the linear operators associated with the $k$-th Fourier projections of \eqref{eq1} and \eqref{eq2}, associated to the linear semigroups $$e^{-tL_{k,\nu}}=S_{\nu}(t)P_k,\ \ e^{-tR_{k,\nu}}=R_{\nu}(t)P_k.$$
For fixed $\nu$ and $k$, $L_{k,\nu}$ and $R_{k,\nu}$ are m-accretive. Notice that $R_{k,\nu}=\nu H_{(ku/\nu)},\ L_{k,\nu}\\ =R_{k,\nu}+\nu|k|^2$. By Theorem \ref{thm:GP}, we have
\begin{align*}\|e^{-tL_{k,\nu}}\|_{L^2\to L^2}&=\|e^{-\nu|k|^2t}e^{-tR_{k,\nu}}\|_{L^2\to L^2}\leq \|e^{-tR_{k,\nu}}\|_{L^2\to L^2}=\|e^{-t\nu H_{(ku/\nu)}}\|_{L^2\to L^2}\\ &\leq e^{-t\nu \Psi(H_{(ku/\nu)})+\pi/2}=e^{-t\nu \Psi_1(ku/\nu)+\pi/2},\ \ \forall\ t\geq0.
\end{align*}

Let us first give the decay rate in terms of $\omega_1(\delta,u)$.
\begin{theorem}\label{thm1}
For $\a>0,\ u\in C(\mathbb{T},\R)$, assume that $\omega_1(\delta,u)\geq C_1\d^{2\a+3} $ for $\d\in(0,1).$ Here $C_1$ is a positive constant.
Then there exist positive constants $ \varepsilon,\ C$ such that for every $\nu>0$ and every integer $k\neq 0$ satisfying $\nu|k|^{-1}\leq 1/2,$ \begin{align}\label{S1}\|S_{\nu}(t)P_k\|_{L^2\to L^2} \leq Ce^{-\varepsilon\widetilde{\la}_{\nu,k}t},\ \|R_{\nu}(t)P_k\|_{L^2\to L^2} \leq Ce^{-\varepsilon\widetilde{\la}_{\nu,k}t},\ \forall\ t\geq 0,\end{align} where $P_k$ denotes the projection to the $k$-th Fourier mode in $x$ and $\widetilde{\la}_{\nu,k}=\nu^{\frac{\a}{\a+2}}|k|^{\frac{2}{\a+2}}$ is the decay rate.
\end{theorem}

\begin{proof}
By the definition, we have $\omega_1(\delta,u)\geq 0 $ is increasing with respect to $\d$ and homogeneous of degree $2$ with respect to $u,\ i.e.\ \omega_1(\delta,Au)=A^2\omega_1(\delta,u)$ for every constant $A\in\R.$ Since $\widetilde{\la}_{\nu,k}=\nu^{\frac{\a}{\a+2}}|k|^{\frac{2}{\a+2}},$ we take $\d=(\nu/\widetilde{\la}_{\nu,k})^{1/2}=(\nu/|k|)^{\frac{1}{\a+2}}\in (0,1)$. Then \begin{align*}
\d\omega_1(\delta,ku/\nu)=\d(|k|/\nu)^2\omega_1(\delta,u)=\d\d^{-2(\a+2)}\omega_1(\delta,u)
 =\omega_1(\delta,u)/\d^{2\a+3}\geq C_1,
 \end{align*}
 By Lemma \ref{Lem: 4}, for $\d=(\nu/\widetilde{\la}_{\nu,k})^{1/2}>0$, we have $\Psi_1(ku/\nu)\geq (\varphi^{-1}( \d\omega_1(\delta,ku/\nu))/\d)^2,$ and
 \begin{align*}\nu \Psi_{1}(ku/\nu)\geq \nu(\varphi^{-1}( \d\omega_1(\delta,ku/\nu))/\d)^2\geq \nu(\varphi^{-1}( C_1)/\d)^2\\ = \nu(\varphi^{-1}( C_1))^2/((\nu/\widetilde{\la}_{\nu,k})^{1/2})^2=(\varphi^{-1}( C_1))^2\widetilde{\la}_{\nu,k}.
 \end{align*}
 Thus,
 \begin{align*}\|S_{\nu}(t)P_k\|_{L^2\to L^2}&=\|e^{-tL_{k,\nu}}\|_{L^2\to L^2}\leq\|R_{\nu}(t)P_k\|_{L^2\to L^2}= \|e^{-tR_{k,\nu}}\|_{L^2\to L^2}\\ &\leq e^{-t\nu \Psi_1(ku/\nu)+\pi/2}\leq e^{-\varepsilon \widetilde{\la}_{\nu,k}t+\pi/2},\ \ \forall\ t\geq0,
 \end{align*}
 where $\varepsilon=(\varphi^{-1}( C_1))^2>0$ is a constant.
 \end{proof}

 The following lemma gives the lower bound of $\omega_1(\delta,u)$ when $u(y)$ is a Weierstrass function.

  \begin{lemma}\label{Lem: 8}
  If $u(y)=\sum\limits_{n=1}^{\infty}a_n\sin(3^ny)$ is a Weierstrass function,\  $a_n\in\R\setminus\{0\},$ and $1\leq |a_n|/|a_{n+1}|\leq 3,$ $ m\in\mathbb{Z},\ m>0,$ then $\omega_1(3^{-m}\pi,u)\geq C^{-1}3^{-3m}a_m^2. $
  \end{lemma}
  \begin{proof}
  We can take $\psi(y)=-\sum\limits_{n=1}^{\infty}\dfrac{a_n}{3^n}\cos(3^ny)$ such that $u(y)=\psi'(y)$. We introduce the difference operator $ \triangle_hf(y)=f(y)-f(y+h),\ \triangle_h^3f(y)=f(y)-3f(y+h)+3f(y+2h)-f(y+3h).$ Noticing that $\triangle_h^3(e^{iny})=e^{iny}(1-e^{inh})^3=ie^{in(y+\frac{3}{2}h)}(2\sin\frac{nh}{2})^3,$ we have $\triangle_h^3\cos({ny})=-\sin(n(y+\frac{3}{2}h))(2\sin\frac{nh}{2})^3,\ \triangle_h^3\psi(y)=\sum\limits_{n=1}^{\infty}\dfrac{a_n}{3^n}\sin(3^n(y+\frac{3}{2}h))(2\sin\frac{3^nh}{2})^3.$
  Noting that for $x,c_1,c_2\in \R,\ h>0,$ we have $ \triangle_h^3(\psi(y)-c_1-c_2y)= \triangle_h^3\psi(y), $ and
  \begin{align*}
  \int_{x-3h}^x(\triangle_h^3\psi(y))^2dy=\int_{x-3h}^x(\triangle_h^3(\psi(y)-c_1-c_2y))^2dy\leq C\int_{x-3h}^{x+3h}(\psi(y)-c_1-c_2y)^2dy,
  \end{align*}
  thus, $\inf\limits_{x\in\R}\int_{x-3h}^x(\triangle_h^3\psi(y))^2dy\leq C\omega_1(3h,u). $ Now we take $h=3^{-m-1}\pi,$ then we can write $\triangle_h^3\psi(y)=f_1(y)+f_2(y) $ with $f_1(y)=\sum\limits_{n=1}^{m-1}\dfrac{a_n}{3^n}\sin(3^n(y+\frac{3}{2}h))(2\sin\frac{3^nh}{2})^3$ and  \begin{align*}&f_2(y)=\sum\limits_{n=m}^{+\infty}\dfrac{a_n}{3^n}\sin\left(3^n\left(y+\frac{3}{2}h\right)\right)\left(2\sin\frac{3^nh}{2}\right)^3
 \\&=\sum\limits_{n=m}^{+\infty}\dfrac{a_n}{3^n}\sin\left(3^ny+\frac{3^{n-m}\pi}{2}\right)\left(2\sin\frac{3^{n-m-1}\pi}{2}\right)^3
 \\&=\dfrac{a_m}{3^m}\cos(3^m y)-\sum\limits_{n=m+1}^{+\infty}\dfrac{8a_n}{3^n}\cos(3^n y).\end{align*}
  Thanks to $1\leq |a_n|/|a_{n+1}|\leq 3,$ we have $|a_n|\leq |a_m|3^{m-n}$ for $1\leq n<m$ and\begin{align*}&|f_1(y)|\leq\sum\limits_{n=1}^{m-1}\dfrac{|a_n|}{3^n}\left|2\sin\frac{3^nh}{2}\right|^3
 \leq\sum\limits_{n=1}^{m-1}\dfrac{|a_n|}{3^n}|{3^nh}|^3\leq\sum\limits_{n=1}^{m-1}\dfrac{|a_m|3^{m-n}}{3^n}|{3^nh}|^3
 \\&=|a_m|3^{m}\sum\limits_{n=1}^{m-1}3^nh^3\leq |a_m|3^{m}\frac{3^mh^3}{2}=|a_m|3^{2m}\frac{(3^{-m-1}\pi)^3}{2}=\dfrac{|a_m|}{3^{m}}\frac{\pi^3}{54}
 .\end{align*}
 Notice that the functions $\{\cos(3^n y)\}_{n\in\mathbb{Z},n\geq m} $ are othogonal in $L^2(x-3h,x)=L^2(x-3^{-m}\pi,x)$  for every $x\in\R,$ and $|a_n|\geq |a_m|3^{m-n}$ for $ n\geq m$. Then we have
 \begin{align*}&\|f_2\|_{L^2(x-3h,x)}^2=\dfrac{|a_m|^2}{3^{2m}}\|\cos(3^m y)\|_{L^2(x-3h,x)}^2+\sum\limits_{n=m+1}^{+\infty}\dfrac{(8|a_n|)^2}{3^{2n}}\|\cos(3^n y)\|_{L^2(x-3h,x)}^2
 \\&=\dfrac{|a_m|^2}{3^{2m}}\frac{3h}{2}+\sum\limits_{n=m+1}^{+\infty}\dfrac{(8|a_n|)^2}{3^{2n}}\frac{3h}{2}\geq \dfrac{|a_m|^2}{3^{2m}}\frac{3h}{2}+\sum\limits_{n=m+1}^{+\infty}\dfrac{(8|a_m|3^{m-n})^2}{3^{2n}}\frac{3h}{2}
\\&=\dfrac{|a_m|^2}{3^{2m}}\frac{3h}{2}\left(1+\sum\limits_{n=m+1}^{+\infty}\dfrac{8^2}{3^{4(n-m)}}\right)
=\dfrac{|a_m|^2}{3^{2m}}\frac{3h}{2}\frac{9}{5}
 .\end{align*}
 Therefore,
 \begin{align*}&\|\triangle_h^3\psi\|_{L^2(x-3h,x)}\geq \|f_2\|_{L^2(x-3h,x)}-\|f_1\|_{L^2(x-3h,x)}\geq\left(\dfrac{|a_m|^2}{3^{2m}}\frac{3h}{2}\frac{9}{5}\right)^{\frac{1}{2}}
 -\left\|\dfrac{|a_m|}{3^{m}}\frac{\pi^3}{54}\right\|_{L^2(x-3h,x)}\\&=\dfrac{|a_m|}{3^{m}}\left(\frac{3h}{2}\frac{9}{5}\right)^{\frac{1}{2}}
 -\dfrac{|a_m|}{3^{m}}\frac{\pi^3}{54}(3h)^{\frac{1}{2}}=\dfrac{|a_m|}{3^{m}}\left(\left(\frac{9}{10}\right)^{\frac{1}{2}}
 -\frac{\pi^3}{54}\right)(3h)^{\frac{1}{2}}\geq \dfrac{|a_m|}{3^{m}}\frac{3}{10}
 (3h)^{\frac{1}{2}}
 .\end{align*}
 Here we used $\left(\frac{9}{10}\right)^{\frac{1}{2}}
 -\frac{\pi^3}{54}\geq \frac{9}{10}-\frac{32}{54}\geq \frac{9}{10}-\frac{6}{10}=\frac{3}{10}. $
 Therefore,
 \begin{align*}&\inf_{x\in \R}\|\triangle_h^3\psi\|_{L^2(x-3h,x)}^2\geq \left(\dfrac{|a_m|}{3^{m}}\frac{3}{10}\right)^2
 (3h)=\left(\dfrac{|a_m|}{3^{m}}\frac{3}{10}\right)^23^{-m}\pi\geq \dfrac{|a_m|^2}{C3^{3m}}
 ,\end{align*} and \begin{align*}&\omega_1(3^{-m}\pi,u)=\omega_1(3h,u)\geq C^{-1}\inf_{x\in \R}\|\triangle_h^3\psi\|_{L^2(x-3h,x)}^2\geq C^{-1}3^{-3m}{|a_m|^2}
 .\end{align*}

 This completes the proof.\end{proof}

Now we are in a position to give some examples of shear flows that induce an enhanced
dissipation time-scale faster than $O(\nu^{-1/3})$.

\begin{lemma}\label{Lem: 9}
If $u(y)=\sum\limits_{n=1}^{\infty}a_n\sin(3^ny)$ is a Weierstrass function,\  $a_n\in\R,\ 3^{-n\a}\leq |a_n|\leq C_03^{-n\a}$ for some constants  $\a\in(0,1),\ C_0>1$ and $1\leq |a_n|/|a_{n+1}|\leq 3$, then there exist positive constants $ \varepsilon,\ C$ such that for every $\nu>0$ and every integer $k\neq 0$ satisfying $\nu|k|^{-1}\leq 1/2,$ \eqref{S1} holds for $\widetilde{\la}_{\nu,k}=\nu^{\frac{\a}{\a+2}}|k|^{\frac{2}{\a+2}}$. \end{lemma}

\begin{proof}
 For every $\d\in(0,1)$, there exists $ m\in\mathbb{Z},\ m>0,$ such that $3^{-m}\pi\leq \d<3^{-m+1}\pi,$ as $|a_m|\geq 3^{-m\a}$. By Lemma \ref{Lem: 8}, we have $\omega_1(\d,u)\geq \omega_1(3^{-m}\pi,u)\geq C^{-1}3^{-3m}{|a_m|^2}\geq C^{-1}3^{-3m}{|3^{-m\a}|^2}=C^{-1}3^{-(3+2\a)m}\geq C^{-1}(\d/(3\pi))^{3+2\a}\geq C^{-1}\d^{3+2\a}. $ Now the result follows from Theorem \ref{thm1}.
 \end{proof}

 \begin{lemma}\label{Lem: 10}If $u(y)=\sum\limits_{n=1}^{\infty}a_n\sin(3^ny)$ is a Weierstrass function,\  $a_n\in\R,\ n^{-\a}\leq |a_n|\leq C_0n^{-\a}$ for some constants $\a\in(1,2),\ C_0>1$ and $1\leq |a_n|/|a_{n+1}|\leq 3$,
 then there exist positive constants $ \varepsilon,\ C$ such that for every $\nu>0$ and every integer $k\neq 0$ satisfying $\nu|k|^{-1}\leq 1/2,$ \eqref{S1} holds for $\widetilde{\la}_{\nu,k}=|k|(\ln (|k|/\nu))^{-\a}$.
 \end{lemma}

 \begin{proof}
 For every $\d\in(0,1)$, there exists $ m\in\mathbb{Z},\ m>0,$ such that $3^{-m}\pi\leq \d<3^{-m+1}\pi,$ thus $m\leq \log_3(\pi/\d)+1\leq C(1-\ln \d).$ As $|a_m|\geq m^{-\a}$, By Lemma \ref{Lem: 8}, we have $\omega_1(\d,u)\geq \omega_1(3^{-m}\pi,u)\geq C^{-1}3^{-3m}{|a_m|^2}\geq C^{-1}3^{-3m}m^{-2\a}\geq C^{-1}\d^{3}(1-\ln \d)^{-2\a}.$  Since $\widetilde{\la}_{\nu,k}=\\|k|(\ln (|k|/\nu))^{-\a},$ we take $\d=(\nu/\widetilde{\la}_{\nu,k})^{1/2}=(\nu/|k|)^{\frac{1}{2}}(\ln (|k|/\nu))^{\frac{\a}{2}}\in (0,1), $ then
 \begin{align*}&\d\omega_1(\delta,ku/\nu)=\d(|k|/\nu)^2\omega_1(\delta,u)=\d(\d^{-2}(\ln (|k|/\nu))^{\a})^2\omega_1(\delta,u)\\ \geq& C^{-1}\d(\d^{-2}(\ln (|k|/\nu))^{\a})^2\d^{3}(1-\ln \d)^{-2\a}
 =C^{-1}(\ln (|k|/\nu))^{2\a}(1-\ln \d)^{-2\a}.
 \end{align*}
 We also have $\d\geq C^{-1}(\nu/|k|)^{\frac{1}{2}},$ $\ln \d\geq (1/2)\ln(\nu/|k|)-C,$ $ 1-\ln \d\leq C-(1/2)\ln(\nu/|k|)=C+(1/2)\ln(|k|/\nu)\leq C\ln(|k|/\nu),$ which implies that $\d\omega_1(\delta,ku/\nu)\geq C_1 $ for an absolute constant $C_1>0.$

 By Lemma \ref{Lem: 4}, for $\d=(\nu/\widetilde{\la}_{\nu,k})^{1/2}>0$, we have $\Psi_1(ku/\nu)\geq (\varphi^{-1}( \d\omega_1(\delta,ku/\nu))/\d)^2,$ and
 \begin{align*}\nu \Psi_{1}(ku/\nu)\geq \nu(\varphi^{-1}( \d\omega_1(\delta,ku/\nu))/\d)^2\geq \nu(\varphi^{-1}( C_1)/\d)^2\\ = \nu(\varphi^{-1}( C_1))^2/((\nu/\widetilde{\la}_{\nu,k})^{1/2})^2=(\varphi^{-1}( C_1))^2\widetilde{\la}_{\nu,k}.
 \end{align*}
 Thus,
 \begin{align*}\|S_{\nu}(t)P_k\|_{L^2\to L^2}&=\|e^{-tL_{k,\nu}}\|_{L^2\to L^2}\leq\|R_{\nu}(t)P_k\|_{L^2\to L^2}= \|e^{-tR_{k,\nu}}\|_{L^2\to L^2}\\ &\leq e^{-t\nu \Psi_1(ku/\nu)+\pi/2}\leq e^{-\varepsilon \widetilde{\la}_{\nu,k}t+\pi/2},\ \ \forall\ t\geq0,
 \end{align*}
 where $ \varepsilon=(\varphi^{-1}( C_1))^2>0$ is a constant. This completes the proof.\end{proof}

 \section*{Acknowledgement}The author would like to thank the
professor Zhifei Zhang for many valuable  suggestions.

\end{CJK*}

\end{document}